\documentclass[english,12pt]{smfart}
  \usepackage[utf8]{inputenc}
  \usepackage{smfthm,smfenum}
  \usepackage{amsmath,amsxtra,amscd,amssymb,latexsym,stmaryrd}
  \usepackage{mathrsfs,dsfont}
  \usepackage{graphicx}
  \usepackage{color}
  \usepackage{pinlabel}

\title[Ultrametrics]{A geometric study of Wasserstein spaces: ultrametrics}

\author{Beno\^{\i}t R. Kloeckner}
\address{Universit\'e de Grenoble I, Institut Fourier\\ CNRS UMR 5582\\ BP 74\\
  38402 Saint Martin d'H\`eres cedex\\ France}
\email{benoit.kloeckner@ujf-grenoble.fr}

\theoremstyle{plain}

\newcommand{\wass}{\mathop{\mathscr{W}}\nolimits}
\newcommand{\dw}{\mathop{\mathrm{W}}\nolimits}
\newcommand{\supp}{\mathop{\mathrm{supp}}\nolimits}
\newcommand{\pr}{\mathop\mathscr{P}\nolimits}
\newcommand{\diam}{\mathop\mathrm{diam}\nolimits}
\newcommand{\udim}{\mathop {\operatorname{\overline{M}-dim}}\nolimits}
\newcommand{\crit}{\mathop \mathrm{crit}\nolimits}
\newcommand{\dd}{\mathrm{d}}
\newcommand{\bcube}{\mathop \mathrm{BC}\nolimits}

\begin{document}
\begin{abstract}
We study the geometry of the space of measures of a compact ultrametric space $X$,
endowed with the $L^p$ Wasserstein distance from optimal transportation. We show
that the power $p$ of this distance makes this Wasserstein space affinely isometric
to a convex subset of $\ell^1$. As a consequence, it is connected by $\frac1p$-Hölder
arcs, but any $\alpha$-Hölder arc with $\alpha>\frac1p$ must be constant.

This result is obtained \emph{via} a reformulation of the distance between
two measures which is very specific to the case when $X$ is ultrametric; however
thanks to the Mendel-Naor Ultrametric Skeleton it has consequences
even when $X$ is a general 
compact metric space. More precisely, we use it to estimate the size of Wasserstein
spaces, measured by an analogue of Hausdorff dimension that is adapted to
(some) infinite-dimensional spaces. The result we get generalizes greatly
our previous estimate that needed a strong rectifiability assumption.

The proof of this estimate involves a structural
theorem of independent interest: every ultrametric space contains large 
co-Lipschitz images of \emph{regular} ultrametric spaces, i.e. spaces of
the form $\{1,\dots,k\}^\mathbb{N}$ with a natural ultrametric.

We are also lead to an example of independent interest: a 
space of positive lower Minkowski dimension, all of whose proper
closed subsets have vanishing lower Minkowski dimension.
\end{abstract}

\thanks{This work was supported by the Agence Nationale de la Recherche, grant
``GMT'' ANR-11-JS01-0011.}

\maketitle

\section{Introduction}

Given a metric space $X$, that we shall always assume to be compact,
one can define its $L^p$ Wasserstein space $\wass_p(X)=(\pr(X),\dw_p)$
as the set of 
its Borel probability measures $\pr(X)$ endowed
with a distance $\dw_p$ defined using optimal transportation (see below for precise definitions).
In some sense, $\wass_p(X)$ can be thought of as a geometric measure theory analogue of $L^p$
space, although its geometry is finely governed by the geometry of $X$ as $\dw_p$ involves
the metric on $X$ in a crucial way; in particular, the great variety of metric spaces
induces a great variety of Wasserstein spaces.
As a consequence, the natural affine structure of $\pr(X)$ (i.e., its affine structure
as a convex in the dual space of continuous functions) is in general only
loosely related to the geometric structure of $\dw_p$.

The links between optimal transportation and geometry have been the object
of a lot of studies in the past decade.
In a series of papers we try to understand what kind of geometric information
on $\wass_p(X)$ can be obtained from given geometric information on $X$. We considered
for example isometry groups and embeddability questions when $X$ is a Euclidean space
\cite{K:WEuclidean} or, with J\'er\^ome Bertrand, a Hadamard space \cite{BK:WHadamard},
and the size of $\wass_p(X)$ when $X$ is (close to
be) a compact manifold \cite{K:Hausdorff}.

Here we consider the case when $X$ is a compact ultrametric space, i.e. satisfies the
following strengthening of the triangular inequality:
\[d(x,z)\leqslant \max(d(x,y),d(y,z)).\]
Examples of compact ultrametric spaces include notably the set of $p$-adic integers $\mathbb{Z}_p$
or more generally
the set $\{1,\dots,k\}^{\mathbb{N}}$ of infinite words on an alphabet with $k$ letters,
endowed with
the distance $d(\bar x, \bar y)=q^{-\min\{i,x_i\neq y_i\}}$ where $q>1$
and $\bar x=(x_1,x_2,\dots)$, $\bar y=(y_1,y_2,\dots)$. We shall call these
examples \emph{regular} ultrametric spaces and denote them by $Y(k,q)$.

\subsection{Embedding in snowflaked $\ell^1$}

Ultrametric spaces are in some sense the simplest spaces in which to do optimal transportation,
thanks to the very strong structure given by the ultrametric inequality. We are
therefore able to give a very concrete description of $\wass_p(X)$.
\begin{theo}\label{theo:coordinates}
If $X$ is a compact ultrametric space, then
$(\pr(X),\dw_p^p)$ is affinely isometric to a convex subset of
$\ell^1$.
\end{theo}
Another way to state this result is to say that $\wass_p(X)$ is
affinely isometric to a convex subset of $\ell^1$ endowed with the
``snowflaked'' metric $\Vert \cdot\Vert_{1}^{1/p}$. 
Note that the existence of such an embedding, both affine and geometrically
meaningful, of $\wass_p(X)$ into a Banach space seems to be quite exceptional; 
the closest case I know of is $\wass_p(\mathbb{R})$, which is isometric
to a subset of increasing functions in $L^p([0,1])$, but even there the isometry is
not affine.
In fact, the absence of correlation between the affine structure of $\pr(X)$ and the geometry
of $\dw_p$ is an important reason for the relevance
of Wasserstein spaces, as the very fact that geodesics in the space of measures (when they exist)
are usually not affine lines
made it possible to define a new convexity
assumption that turned out to be very successful, see notably \cite{McCann}. In this
sense, Theorem \ref{theo:coordinates} is a negative result: when $X$ is
ultrametric there is little more to $\wass_p(X)$ than to $\pr(X)$, as far as 
we are concerned with notions which are affine in nature (e.g. convexity).
We shall see, however, that this result has nice consequences.

As is well-known,
snowflaked metrics are geometrically very disconnected (all rectifiable curves are constant)
and this affects the geometric connectivity of Wasserstein space.
\begin{coro}\label{coro:coordinates}
If $X$ is a compact ultrametric space, $\wass_1(X)$ is a geodesic space,
and for $p>1$, $\wass_p(X)$
is connected by $\frac1p$-Hölder arcs but any $\alpha$-Hölder arc
with $\alpha>\frac1p$ must be constant.
\end{coro}

\subsection{Size estimates}

While their strong structural properties make ultrametric spaces feel quite
easy to deal with, they also happen to be ubiquitous, as shown by the
ultrametric skeleton Theorem of Mendel and Naor \cite{Mendel-Naor:skeleton,Mendel-Naor:ultrametric}:
very roughly,
any metric space contains large almost ultrametric parts. This powerful result
enables us to control very precisely the size of very general Wasserstein spaces.
\begin{theo}\label{theo:size}
Given any compact metric space $X$ (not necessarily ultrametric), we have
\[\crit_\mathscr{P} \wass_p(X) \geqslant \dim X.\]
\end{theo}
Here, $\dim$ denotes the Hausdorff dimension, and $\crit_\mathscr{P}$
is the power-exponential critical parameter introduced in \cite{K:Hausdorff},
which is an extension of Hausdorff dimension that distinguishes
some infinite-dimensional spaces. This bi-Lipschitz invariant is constructed simply
by replacing the terms $\varepsilon^s$ by $\exp(-\varepsilon^{-s})$ in the definition
of Hausdorff dimension. In particular,
the above result implies that to cover $\wass_p(X)$ one needs
at least \emph{very roughly} $\exp(\varepsilon^{-\dim X})$ balls of radius $\varepsilon$.

\begin{rema}
It is in fact possible to give an elementary proof of Theorem \ref{theo:size}
that does not use ultrametric spaces. We shall give such a proof
in Section \ref{sec:alt}, but we feel 
the ultrametric proof has its own worth as it shows a way to use
the ultrametric skeleton theorem and could apply more generally
(notably to estimate the size of other large spaces, as spaces of closed
subsets, spaces of Hölder functions, etc.).

In fact, the non-ultrametric proof was only found some time after submission
of the first version of this article, and the ultrametric skeleton theorem
played a key role in the author's mind when thinking about the whole issue.
\end{rema}

As we proved in \cite{K:Hausdorff} that $\crit_\mathscr{P} \wass_p(X)$ is at 
most the upper
Min\-kow\-ski dimension $\udim X$ of $X$, the following results follow at once
from Theorem \ref{theo:size}.

\begin{coro}
Let $X$ be any compact metric space (not necessarily ultrametric);
if $\dim X= \udim X=d$, then $\crit_\mathscr{P} \wass_p(X)=d$.
\end{coro}
This greatly generalizes two of the main results of \cite{K:Hausdorff}, focused on $p=2$
and where either
instead of $\dim X=d$ we had to assume  the much stronger assumption that $X$ contains a bi-Lipschitz
image of $[0,1]^d$, or we could only conclude a much weaker lower bound on the size of $\wass_2(X)$
(even weaker than $\crit_\mathscr{P}\wass_2(X)>0$).

\begin{coro}
Let $X,X'$ be two compact metric spaces. If $\dim X > \udim X'$, then there is
no bi-Lipschitz embedding $\wass_p(X)\to \wass_{p'} (X')$, for any $p,p'\in [1,\infty)$.
\end{coro}

Let us also note that the Mendel-Naor ultrametric skeleton Theorem readily
implies that for all $\varepsilon>0$, $\wass_p(X)$ contains a subset $S$
with $\crit_\mathscr{P} S \geq(1-\varepsilon)\dim X$ 
that embeds in an ultrametric space with distortion
$O(1/\varepsilon)$.

It would be more natural 
to replace the inequality on dimensions in Theorem \ref{theo:size}
by $\dim X > \dim Y$, but
our method cannot give that stronger statement. This seems inevitable
since the Hausdorff
dimension of $X$ gives no upper bound on the critical parameter
of its Wasserstein spaces.
\begin{prop}\label{prop:counterex}
There is an ultrametric space $X$ such that $X$ is countable
(in particular $\dim X = 0$)
but $\crit_\mathscr{P} \wass_p(X) = +\infty$ for all $p$.
\end{prop}

\subsection{Structure of ultrametric space}

In addition to the Mendel-Naor theorem, one important ingredient in the
first proof
of Theorem \ref{theo:size} is a structural result that seems of interest
in itself, according to which
every compact ultrametric space contains (in a weak sense) a large regular part,
which is easier to deal with.
\begin{theo}\label{theo:structure}
Given any compact ultrametric space $X$ and any $s<\dim X$,
there is a regular ultrametric space $Y=Y(k,q)$ of Hausdorff dimension
at least $s$ and a co-Lipschitz map
$\varphi : Y\to X$.
\end{theo}
By a co-Lipschitz map we mean that for some $c>0$ and all
$a,b\in Y$ one has
\[d(\varphi(a),\varphi(b))\geqslant c\cdot d(a,b).\]

This shows that, a bit like ultrametric spaces are ubiquitous in metric spaces,
\emph{regular} ultrametrics are ubiquitous in ultrametrics (and therefore
in metric spaces). The fact that we only
obtain a co-Lipschitz map is a strong limitation, but this is sufficient for our present purpose.

\subsection{Organization of the article}

In the next section, we introduce briefly some classical definitions
and facts concerning Wasserstein spaces, ultrametric spaces and
we recall some properties of critical parameters.
In Section \ref{sec:coordinates}, we prove Theorem \ref{theo:coordinates}
and corollary \ref{coro:coordinates}. Section \ref{sec:structure}
is devoted the purely ultrametric Theorem \ref{theo:structure},
while Sections \ref{sec:size} and \ref{sec:alt} give the two proofs of 
Theorem \ref{theo:size}.
Last, Section \ref{sec:examples} gives two examples motivating
the dimension hypothesis in Theorem \ref{theo:size}, notably proving
Proposition \ref{prop:counterex}.

\section{Preliminaries}

\subsection{Wasserstein spaces}\label{sec:Wasserstein}

We limit ourselves in this short presentation
to the case of a compact metric spaces $X$ whose distance is denoted by $d$.
The set of Borel probability
measures $\pr(X)$ is naturally endowed with a set of ``Wasserstein'' distances that echo
the distance of $X$: for any $p\in[1,+\infty)$, one sets
\[\dw_p(\mu,\nu) = \left(\inf_{\Pi\in\Gamma(\mu,\nu)} 
  \int_{X\times X} d(a,b)^p \,\Pi(\dd a\,\dd b)\right)^{\frac1p}\]
where $\Gamma(\mu,\nu)$ is the set of \emph{transport plans}, or \emph{coupling}
of $(\mu,\nu)$, that is to say the set of measures $\Pi$ on $X\times X$
that projects on each factor as $\mu$ and $\nu$:
\[\Pi(A\times X)=\mu(A),\quad \Pi(X\times B)=\nu(B)\quad \forall\mbox{ Borel } A,B\subset X.\]
In other words, a transport plan specifies a way of allocating mass distributed according to
$\mu$ so that it ends up distributed according to $\nu$; its \emph{$L^p$ cost}
\[c_p(\Pi):=\inf_{\Pi\in\Gamma(\mu,\nu) }
  \int_{X\times X} d(a,b)^p \,\Pi(\dd a\,\dd b)\]
is the total cost of this allocation if one assumes that
allocating a unit of mass from a point to a point $d$ away costs $d^p$, and
$\dw_p(\mu,\nu)^p$ is the least possible cost of a transport plan.

It is easily proved (see e.g. \cite{Villani:book} for this and much more) that
the infimum is realized by what is then called an \emph{optimal} transport plan,
that $\dw_p$ is indeed a distance and that it metricizes the weak topology.
We shall denote by $\wass_p$ the metric space $(\pr(X),\dw_p)$ and call it
the ($L^p$) \emph{Wasserstein space} of $X$.

We shall need very little more from the theory of optimal transportation, let
us only state two further facts.

First, an easy consequence of ``cyclical monotonicity'' is that
if $X$ is a metric tree (or its completion), $e$ is an edge of $X$, $\Pi$ is an optimal transport plan
between measures $\mu,\nu\in\wass_p(X)$ supported outside $e$,
and $X_1,X_2$ are the connected components of $X\setminus e$, then
\[\Pi(X_1\times X_2) = \max(\mu(X_1)-\nu(X_1),0);\]
in other words, no more mass is moved through an edge than strictly necessary.

Second, the concept of \emph{displacement interpolation} shall prove convenient.
If $X$ is a geodesic space, $\mu,\nu\in\wass_p(X)$ and $\Pi$ is an optimal transport
plan (implicitly, for the cost $c_p$ and from $\mu$ to $\nu$), then it is known
that there is a probability measure $\pi$ on the set of constant speed
geodesics $[0,1]\to X$ such that if one draws a random geodesic $\gamma$
with law $\pi$, the random pair $(\gamma(0),\gamma(1))$
of its endpoints has law $\Pi$. In other words, if $e_t:\gamma\mapsto\gamma(t)$
is the specialization map, $\Pi=(e_0,e_1)_\#\pi$. The measure $\pi$ is
called an optimal \emph{dynamical} transport plan. The interest
of this description is that $\wass_p(X)$ is geodesic, and all its
geodesics have the form $(e_{t\#}\pi)_{t\in[0,1]}$ for some optimal dynamical transport plan $\pi$.

Now, if $X$ is a metric tree, let say that two geodesics $\gamma_1$, $\gamma_2$
are \emph{antagonist} if they both follow some edge $e$, in opposite directions.
Then an optimal dynamical plan $\pi$ between any two measures is supported on a set
without any pair of antagonist geodesics. This is of course closely linked to the cyclical monotonicity.

\subsection{Critical parameters}

Critical parameters where introduced in \cite{K:Hausdorff} as bi-Lipschitz invariants
similar to Hausdorff dimension but that can distinguish between
some infinite-dimensional spaces, notably many Wasserstein spaces.

We shall not recall their construction here, but only a few facts that
we shall use in the sequel. First, to define a critical parameter
one needs a so-called \emph{scale}, a family of functions playing the role
played by $(r\mapsto r^s)$ in the definition of Hausdorff
dimension. We restrict here to the power-exponential scale
$\mathscr{P}=(r\mapsto\exp(-1/r^s))_{s>0}$ and denote
the corresponding critical parameter by $\crit_\mathscr{P}$.

Then Frostman's Lemma (see e.g. \cite{Mattila}) 
gives a characterization of $\crit_\mathscr{P}$, from
which we extract the following conditions:
\begin{itemize}
\item if $\crit_\mathscr{P}>s$ then there is $\mu\in\pr(X)$ and $C>0$ such that
      for all $x\in X$ and all $r>0$, $\mu(B(x,r)) \leq C \exp(-1/r^s)$;
\item if there is a measure $\mu\in\pr(X)$ as above, then $\crit_\mathscr{P}\geq s$.
\end{itemize}
In other words, $X$ has large critical parameter if it supports a very spread out measure.
$\crit_\mathscr{P} X$ is zero when $X$ has finite Hausdorff dimension, but turns out
to be non-zero for many interesting infinite-dimensional spaces.

The second important fact is that $\crit_\mathscr{P}$ can only increase under co-Lipschitz maps: if
$f:Y\mapsto X$ satisfies $d(f(a),f(b))\geq c\cdot d(a,b)$ for some $c>0$ and all $a,b\in Y$,
then $\crit_\mathscr{P}(X)\geq\crit_\mathscr{P}(Y)$. To prove lower bounds on critical
parameters, our strategy will be to find in our space of interest co-Lipschitz images
of spaces supporting a well spread-out measure.

To do that, we will need spaces on which such measures are easy to construct.
We shall use the \emph{Banach cubes} 
\[\bcube((a_n)_n) :=\left\{(x_n)\in\ell^1 \,\middle|\, 
  0\leq x_n \leq a_n \,\forall n\in\mathbb{N} \right\}\]
defined for any $\ell^1$ positive sequence $(a_n)$.
In \cite{K:Hausdorff}, $\bcube((a_n)_n)$ was denoted $\bcube([0,1],1,(a_n)_n)$
as a more general family of Banach cubes was defined. We have
\begin{equation}
\crit_\mathscr{P}\bcube((n^{-\alpha})) = \frac{1}{\alpha-1}
\label{eq:BC}
\end{equation}
for all $\alpha>1$; we shall only give a sketch of the
proof, as it follows the same lines as the proof of the Hilbertian
version of this estimate, given in full details in \cite{K:Hausdorff} (section 4,
see also Proposition 8.1 page 232).

\begin{proof}[Sketch of proof of \eqref{eq:BC}]
The upper bound is obtained by bounding the upper Minkowski critical parameter,
i.e. by bounding from above the number of balls of radius $\varepsilon$
needed to cover $\bcube((n^{-\alpha}))$. Let $L=L(\varepsilon/2)$ be the first integer
such that 
\[\sum_{n>L} n^{-\alpha} \leq \frac\varepsilon 2\]
and, for each $n\leq L$, consider a minimal set of points $(x^i_n)_i$ on
$[0,n^{-\alpha}]$ such that every point of this interval is at distance at most
$\varepsilon/(C n \log^2 n)$ of one of the $x^i_n$, where $C$
is such that $\sum_1^\infty (C n \log^2 n)^{-1} \leq 1/2$.
Then, any point $\bar x = (x_1,x_2,\dots)\in \bcube((n^{-\alpha}))$
is a distance at most $\varepsilon$ from one of the points
\[(x_1^{i_1},x_2^{i_2},\dots,x_L^{i_L},0,0,\dots)\]
Then, an estimate of the number of such points shows that
\[\crit_\mathscr{P}\bcube((n^{-\alpha})) \leq \frac{1}{\alpha-1}.\]

The lower bound is obtained using Frostman's Lemma. We consider
the uniform probability measure $\lambda_n$ on $[0,n^{-\alpha}]$
and the measure $\mu:= \otimes_n \lambda_n$ on $\bcube((n^{-\alpha}))$.
Then, one can prove that for all $\beta < (\alpha-1)^{-1}$, there is a constant
$C$ such that for all $\bar x\in \bcube((n^{-\alpha}))$ and all $r\leq 1$,
\[\log\mu(B(\bar x, r)) \leq -C\frac1{r^\beta}\]
(see pages 217-218 in \cite{K:Hausdorff}).
Frostman's Lemma (Proposition 3.4 of \cite{K:Hausdorff}) then ensures
\[\crit_\mathscr{P}\bcube((n^{-\alpha})) \geq \frac{1}{\alpha-1}.\]
\end{proof}

\subsection{Ultrametric spaces}\label{sec:ultrametric}

According to the Ultrametric skeleton Theorem \cite{Mendel-Naor:skeleton,Mendel-Naor:ultrametric},
given any compact metric space
$X$ and any $s<\dim X$, there is an \emph{ultrametric} space $X'$
of dimension at least $s$ and a bi-Lipschitz embedding $X'\to X$ (moreover,
the distortion of this bi-Lipschitz embedding is $O(\dim X-s)^{-1})$). This simply
stated result is very powerful, see \cite{Naor:Ribe}.

The other, older and classical fact we shall need about ultrametric spaces
is their description in terms of trees (see e.g. \cite{Naor:Ribe} Section 8.1 or \cite{GNS}). 

By a \emph{tree} we mean
a simple graph $T$ with vertex set $V$ and edge set $E$, which is connected and without
cycle; it can be infinite but is assumed to be locally finite. 
A tree is \emph{rooted} if it has a distinguished
vertex $o$, which is not a leaf, and called the root. A vertex $v\neq o$ has a
unique \emph{parent} $v^*$, defined as the neighbor of $v$ closer to $o$ than $v$; $v$ is
then said to be a \emph{child} of $v^*$. Each edge has a natural orientation, from
parent to child: $(v^*v)$ is said to be a positive edge.
A \emph{height function} is a function $h:V\to[0,+\infty)$
that is decreasing: $h(v)<h(v^*)$ for all $v\neq o$ (note that our trees
have the root on top). A \emph{synchronized rooted tree} (SRT for short)
is a rooted tree endowed with a height function such that for any maximal (finite or infinite)
oriented path $o, v_1, v_2,\dots$, we have $\lim h(v_n)=0$ (in particular, all leaves
have height $0$). The metric realization
of a SRT $T$ is the metric space obtained by taking a segment of length $h(x)-h(y)$
for each positive edge $(xy)$ and gluing them according to $T$.
It is still denoted by $T$, and the height function can be extended linearly on edges
and continuously to the metric completion $\bar T$ of $T$; this extension is
still denoted by $h$. By construction,
$h^{-1}(0)$ is the union of all leafs of $T$ and of $\bar T\setminus T$.
It is not hard to check that $h^{-1}(0)$, endowed with the restriction of the
metric of $\bar T$, is ultrametric.
We can now state the description alluded to above.

\begin{center}\begin{minipage}{.9\textwidth}
\textit{Any compact ultrametric space can be isometrically identified with
the level $h^{-1}(0)$ of the completion of a metric SRT $T$.}
\end{minipage}\end{center}

For each vertex $v$
of $T$, the set of points in $X$ that can be reached by an oriented path
from $v$ is a metric ball of $X$, denoted by $X_v$ and of diameter
$2h(v)$. All balls of $X$ are of the form $X_v$ for some $v$, e.g.
$X=X_o$.

The proof of the above folkloric fact is not difficult, and can be found up
to little notational twists in the references cited above: one simply
use the ultrametric inequality to partition $X$ into maximal proper balls,
which will be identified with the children of $o$, and then proceed recursively.

\section{$\ell^1$ coordinates}\label{sec:coordinates}

The goal of this Section is to prove Theorem \ref{theo:coordinates}:
given a compact ultrametric space $X$ and $p\in[1,\infty)$, we want
to construct an affine isometry from $(\pr(X),\dw_p^p)$ to a
convex subset of $\ell^1$, that is a map
$\varphi:\pr(X)\to\ell^1$ with convex image and such that
\[\varphi(t\mu+(1-t)\nu)=t\varphi(\mu)+(1-t)\varphi(\nu) 
  \quad\forall \mu,\nu\in\pr(X),\forall t\in[0,1]\]
and
\[\Vert\varphi(\mu)-\varphi(\nu)\Vert_1 = \dw_p(\mu;\nu)^p \quad\forall \mu,\nu\in\pr(X).\]
Up to coefficients, this map is simply constructed by mapping a measure to the collection
of masses it gives to balls of $X$.

\subsection{A formula for the Wasserstein distance}

Let $T$ be a metric SRT such that $X=h^{-1}(0)\subset\bar T$ as explained in §\ref{sec:ultrametric}.
Then $\pr(X)$ is the subset of $\pr(\bar T)$ made of measures concentrated on $X$,
and $\dw_p$ is the restriction to $\pr(X)$ of the Wasserstein metric on $\bar T$,
also denoted by $\dw_p$. 
Given a geodesic $\gamma$ of $\bar T$, let $E(\gamma)$ be the set of edges
through which $\gamma$ runs and let $v_+(\gamma)$ be the topmost vertex on $\gamma$.
Given $e=(v^*v)\in E$, set $\delta h^p(e)=h(v^*)^p-h(v)^p$. Last,
given $e\in E$ let $e_-$ be its lower vertex and $\Gamma(e)$ 
be the set of geodesics going through $\gamma$
in any direction (not to be confused with a set of optimal transport plan!).

\begin{lemm}\label{lemm:coordinates}
For all $\mu,\nu\in\pr(X)$, we have
\begin{equation}
\dw_p(\mu,\nu)^p= 2^{p-1}\sum_{v\neq o\in V} \delta h^p(v^*v) |\mu(X_{v})-\nu(X_{v})|.
\label{eq:coordinates}
\end{equation}
\end{lemm}

\begin{proof}
Let $\Pi\in\Gamma(\mu,\nu)$
be an optimal transport plan. For any vertex $v\in T$, the components
of $\bar T\setminus (vv^*)$ intersect with $X$ along $X_v$ and $X_v^c:=X\setminus X_v$.
As noted in §\ref{sec:Wasserstein},
optimality implies that $\Pi(X_v\times X_v^c)=\max(0,\mu(X_v)-\nu(X_v))$
and $\Pi(X_v^c\times X_v)=\max(0,\nu(X_v)-\mu(X_v))$: the total amount of mass
that moves between $X_v$ and its complement is $|\mu(X_v)-\nu(X_v)|$.

Let $\pi$ be an optimal dynamical transport plan on $T$ such that $\Pi=(e_0,e_1)_\#\pi$.
Then for all geodesic $\gamma$ between two points of $X$, we have
\[d(\gamma(0),\gamma(1))^p=(2h(v_+(\gamma)))^p=2^{p-1}\sum_{e\in E(\gamma)}\delta h^p(e)\]
from which it follows
\begin{eqnarray}
\dw_p(\mu,\nu)^p=c_p(\Pi) &=& \int d(\gamma(0),\gamma(1))^p \,\pi(\dd \gamma) \nonumber\\
  &=& \int 2^{p-1}\sum_{e\in E(\gamma)}\delta h^p(e) \,\pi(\dd \gamma)\nonumber\\
  &=& 2^{p-1}\sum_{e\in E} \delta h^p(e) \pi(\Gamma(e)) \nonumber\\
  &=& 2^{p-1}\sum_{e\in E} \delta h^p(e) |\mu(X_{e_-})-\nu(X_{e_-})|.
\end{eqnarray}
which is \eqref{eq:coordinates}.
\end{proof}

\subsection{Proof of Theorem \ref{theo:coordinates}}

Identify the set of non-root vertices $V\setminus \{o\}$ of $T$ with
the positive integers, so that $\ell^1=L^1(V\setminus \{o\})$ (with the counting measure).
Let $\varphi:\pr(X)\to \ell^1$ be defined by
\[\varphi(\mu)=\left(2^{p-1}\delta h^p(v^*v) \mu(X_v)\right)_{v\neq o\in V}.\]
Lemma \ref{lemm:coordinates} shows that $\varphi$ is an isometric embedding,
and it is obviously affine. It extends as an affine embedding from the space
of signed measures (in which $\pr(X)$ is convex) to $\ell^1$, so that $\varphi$
has convex image: Theorem \ref{theo:coordinates} is proved.

For Corollary \ref{coro:coordinates}, recall the fact that
a convex subset of $\ell^1$ is geodesic (thus connected by Lipschitz arcs)
and that for any metric space $(Y,d)$, $d^\beta$ defines a distance
without any non-constant Lipschitz curve whenever $\beta\in(0,1)$
(which is folklore, see e.g. Lemma 5.4 in \cite{BK:WHadamard},  for a simple proof).

Apply this to 
\[\dw_p^{\frac1\alpha}=(\dw_p^{p})^{\frac1{\alpha p}}\]
with $\beta=1/\alpha p<1$ where, thanks to Lemma \ref{lemm:coordinates}, we know that $\dw_p^p$ 
is a distance: we get that $\dw_p^{1/\alpha}$ is a distance without non-constant Lipschitz curves.
This means precisely that $\dw_p$ has no non-constant $\alpha$-H\"older curves.

\section{Regular ultrametric parts in ultrametric spaces}\label{sec:structure}

Let us now prove Theorem \ref{theo:structure}. We are given a compact
ultrametric space $X$ and $s<\dim X$, and we look for a regular ultrametric
space $Y=Y(k,q)$ with $\dim Y\geq s$ and a co-Lipschitz map $\varphi:Y\to X$.

Let $T$ be a SRT such that $X=h^{-1}(0)$, and fix $\varepsilon>0$ such that
$s':=s+\varepsilon<\dim X$. We assume, up to a dilation of the metric, that $\diam X=1$.

\subsection{Defining $k$ and $q$}

By Frostman's Lemma, there exists on $X$
a probability measure $\mu$ and a constant $C$ such that for all $x\in X$ and all $r$,
we have 
\begin{equation}
\mu(B(x,r))\leq C r^{s'}.
\label{eq:Frostman}
\end{equation}
Choose an integer $k$ such that $3k>3^{s'/\varepsilon}$ and $3k>2C$, and let $q=(3k)^{1/s'}$.
This choice of $q$ ensures that $Y(3k,q)$ has Hausdorff (and Minkowski) dimension equal to $s'$,
and the first bound on $k$ ensures that $Y=Y(k,q)$ has dimension at least $s$. The second bound
on $k$ is a technicality to be used later.

The SRT of $Y$ is the $k$-regular rooted
tree with height function $h(v)=\frac12 q^{-n}$ whenever $v$ is at combinatorial
distance $n$ from the root.
Our strategy is now to change slightly the metric on $X$,
then use $\mu$ to transform $T$ into a regular tree, while controlling both distances
from above and dimension from below.

\subsection{Changing heights}

The following Lemma is folklore.
\begin{lemm}
There is an ultrametric $d'$ on $X$ such that
all balls of $X$ have diameters of the form $q^{-n}$ with integer $n$, and
\[d(x,y)\leq d'(x,y) < q\cdot d(x,y) \quad \forall x,y\in X.\]
\end{lemm}
Said otherwise, this lemma ensures that we can assume that 
$h$ takes only the values $\frac12 q^{-n} \ (n\in\mathbb{N})$
on vertices of $T$.

\begin{proof}
Consider the SRT obtained from $T$ by changing $h(v)$ into
the smallest $\frac12 q^{-n}$ larger than $h(v)$. The $0$ level
of $h'$ is still naturally identified with $X$, the induced distances are no smaller
than the original one, and they are larger by at most a factor $q$.
\end{proof}

Note that the measure
$\mu$ still satisfies \eqref{eq:Frostman} in the new metric with
the same $C$. From now on, we assume that $X$ satisfies the conclusion of the above lemma.

\subsection{Regrouping branches}

Consider now the children $v_1,v_2,\dots,v_j$ of the root, and let
$V_1,\dots,V_{J}$ be a partition of $\{v_1,\dots,v_j\}$ into
at least two sets of consecutive vertices such that 
\[\frac12 C q^{-s'} \leq \sum_{v\in V_I} \mu(X_v) \leq \frac32 C q^{-s'}.\]
Such a partition exists thanks to the second bound on $k$, which
ensures that $C q^{-s'} < \frac12$.
Let $T^1$ be the SRT obtained from $T$ by (see figure \ref{fig:regrouping}): 
\begin{itemize}
\item adding a degree two vertex at height $q^{-s'}$ on the edge $(ov_i)$ whenever $h(v_i)<q^{-s'}$,
\item reassigning the name $v_i$ to this new added vertex,
\item then for each $I\in\{1,\dots,J\}$, merging all $v_i\in V_I$ into a new vertex of height
      $q^{-s'}$, whose children are the union of all children of the $v_i\in V_I$.
\end{itemize}

\begin{figure}[htp]\begin{center}
\labellist
\small \hair 3pt
\pinlabel $o$ [b] at 60 97
\pinlabel $o$ [b] at 213 97
\pinlabel $o$ [b] at 365 97
\pinlabel $1$ at 139 96
\pinlabel $q^{-1}$ at 139 65
\pinlabel $q^{-2}$ at 139 35
\pinlabel $q^{-3}$ at 139 4
\pinlabel $1$ at 291 96
\pinlabel $q^{-1}$ at 291 65
\pinlabel $q^{-2}$ at 291 35
\pinlabel $q^{-3}$ at 291 4
\endlabellist
\includegraphics[width=.9\linewidth]{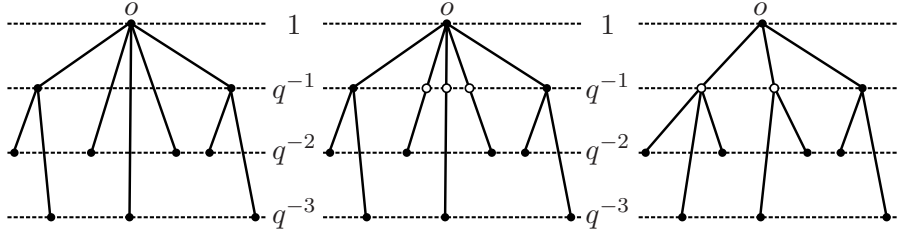}
\caption{The regrouping process: left the original tree (partially represented),
  center the tree with 
  added vertices, right the tree with regrouped branches.}\label{fig:regrouping}
\end{center}\end{figure}

Below depth $1$, the tree is unchanged and there is therefore
a natural identification of $X$ with the level $0$ in $T^1$, and the
distance induced by $T^1$ is no larger than the original one (it can be much smaller
for some pair of points, and this is why we will only obtain a co-Lipschitz map).
Another way to put it is that there is a bijective and $1$-Lipschitz map $f^1:X\to X^1$
where $X^1$ is the level $0$ of $T^1$. Moreover $\mu^1:=f^1_\#\mu$ is a probability
measure on $X^1$ satisfying both 
\[\mu^1(X^1_v)\leq C (\diam X^1_v)^{s'} = C \cdot q^{-ns'}\]
when $v$ has height $q^{-n}$ with $n>1$ and
\[\frac12 C q^{-s'} \leq \mu^1(X^1_v) \leq \frac32 C q^{-s'}\]
when $v$ is a child of the root (i.e., has height $q^{-1}$).

We can then inductively construct a sequence $T^2, T^3, \dots, T^k, \dots$
of SRT by performing the same 
regrouping process at depth $k$ (i.e., the vertices of height $q^{-(k-1)}$
play the role played above by $o$). First this construction ensures that in $T^k$,
a child $v$ of any vertex $v^*$ of height $h^k(v^*)=q^{-n}$ for any $n< k$ must have
height $h^k(v)=q^{-(n+1)}$.
We also get a system of bijective $1$-Lipschitz maps $f^k:X^{k-1}\to X^k$
where $X^k$ is the ultrametric space defined by $T^k$, and probability measures
$\mu^k$ that satisfy 
\[\mu^k(X^k_v)\leq C (\diam X^k_v)^{s'} = C \cdot q^{-ns'}\]
when $v$ has height $q^{-n}$ with $n>k$ and
\[\frac12 C q^{-ns'} \leq \mu^k(X^k_v) \leq \frac32 C q^{-ns'}\]
when $v$ has height $q^{-n}$ with $n\leq k$.

Since $T_k$ and $T_{k+1}$ are isomorphic up to depth $k$, we get a limit
SRT $T^\infty$, defining an ultrametric space $X^\infty$
and there is a $1$-Lipschitz map $f:X\to X^\infty$ obtained by composing all
$f^k$. This map needs not be bijective, because some distances may have 
been reduced to zero in the process, collapsing some points together; but
it certainly is onto.

Moreover, the probability measure $f^\infty_\#\mu=\mu^\infty$ satisfies
\[\frac12 C q^{-ns'} \leq \mu^\infty(X^\infty_v) \leq \frac32 C q^{-ns'}\]
whenever $v$ has depth $n$.
This ensures that any vertex in $T^\infty$ (except possibly the root) has at least
\[\frac{\frac12 C q^{-ns'}}{\frac32 C q^{-(n+1)s'}}=\frac{q^{s'}}{3} = k\]
children. In particular, $T^\infty$ has a subtree isomorphic to the SRT of $Y$, so that
there is an isometric embedding $g:Y\to X^\infty$.

Composing $g$ with a right inverse of $f^\infty$ (which exists at worst in the measurable category
in virtue of a classical selection theorem), we get a co-Lipschitz map $Y\to X$, which proves Theorem
\ref{theo:structure}
(recall that our choice of parameter ensures $\dim Y\geq s$).

\section{Size of Wasserstein spaces}\label{sec:size}

\subsection{The case of regular ultrametric spaces}

In the previous Section, we saw that ultrametric spaces contain
co-Lipschitz images of large regular ultrametric spaces. To prove Theorem
 \ref{theo:size} we therefore have mainly left to estimate the size
of Wasserstein spaces of regular ultrametric spaces.

\begin{prop}
Given any regular compact ultrametric space $Y=Y(k,q)$ and any $p\geq 1$, we have
\[\crit_\mathscr{P} \wass_p(Y) = \log_q(k)=\dim Y.\]
\end{prop}

\begin{proof}
Since the upper Minkowski dimension of $Y$ is equal to its Hausdorff dimension,
the upper bound $\crit_\mathscr{P} \wass_p(Y)\leq \dim Y$ is given by Proposition 7.4 in 
\cite{K:Hausdorff}.
To prove the lower bound, we shall embed large Banach cubes in $\wass_p(Y)$. Fix any 
positive $\varepsilon$.

Label as usual the vertices of the SRT for $Y$ (i.e, its balls) by the finite words on the 
letters $\{1,\dots, k\}$. A vertex $v=i_1i_2\dots i_n$ is said to have \emph{depth} $n$,
has height $\frac12q^{-n}$, and there are $k^n$ of them. Let $V'$ be the set of
vertices $v=i_1\dots i_n$ such that $i_n<k$, and define a map $\varphi:[0,1]^{V'}\to\pr(Y)$
as follows. The measure $\mu=\varphi((a_v)_{v\in V'})$ is determined 
by the weight it gives to the balls $Y_v$ ($v\in V$) of $Y$, which we define
recursively on the depth to be:
\begin{itemize}
\item $\mu(Y)=1$,
\item $\displaystyle\mu(Y_v)=\mu(Y_{v^*})\cdot\frac{1+\varepsilon a_v}{k}$ when $v\in V'$,
\item $\displaystyle\mu(Y_v)=\mu(Y_{v^*})-\sum_{w^*=v^*, w\neq v} \mu(Y_w)$ when $v=i_1\dots i_{n-1}k$.
\end{itemize}
In other words, at each level we split mass almost equally between the children, allowing it to be
slightly larger than average for the $k-1$ first children and consequently slightly smaller
for the last one.

The first and second items are mandatory to get a well defined probability measure,
taking $\varepsilon$ small enough ensures that all these values are positive, and this construction
ensures that 
\[\mu(Y_v) \geq \left(\frac{1-(k-1)\varepsilon}{k}\right)^n\]
whenever $v$ has depth $n$.

From \eqref{eq:coordinates} in Section \ref{sec:coordinates}, we deduce that
for all $a=(a_v), b=(b_v)\in [0,1]^{V'}$: 
\[\dw_p(\varphi(a),\varphi(b))^p \geq 
  C \sum_{v\in V'} q^{-pn}\left(\frac{1-(k-1)\varepsilon}{k}\right)^n |a_v-b_v| \]
where the positive constant $C$ depends on $q,k,p,\varepsilon$ and $n=n(v)$ is the depth.

Since there are $(k-1)k^{n-1}$ vertices of depth $n$ in $V'$, if we identify
the vertices with the positive integers in a way that makes the depth function $n$ non-decreasing,
we can identify the sequence
\[\left(q^{-p}\frac{1-(k-1)\varepsilon}{k}\right)^n(v)\]
where $v$ runs over the vertices with an integer-indexed sequence $(a_m)$ where
\[a_m =\Theta\left(q^{-p}\frac{1-(k-1)\varepsilon}{k} \right)^{\log_k m}
       = \Theta \left(m^{\log_k(\frac{q^{-p}}{k}-O(\varepsilon))}\right).\]
It follows that
there is a co-Lipschitz map from $\bcube((m^{-\alpha})_m)$ to $(\pr(Y),\dw_p^p)$
with
\[\alpha = 1+p\frac{\ln q}{\ln k} -O(\varepsilon) = 1+\frac{p}{\dim Y} -O(\varepsilon).\]
As a consequence,
\[\crit_\mathscr{P} (\pr(Y),\dw_p^p) \geq \crit_\mathscr{P} \bcube((n^{-\alpha}))
  = \frac{1}{\alpha-1}\]
and letting $\varepsilon$ go to $0$, we have $\crit_\mathscr{P} (\pr(Y),\dw_p^p)\geq \dim Y/p$
from which $\crit_\mathscr{P} \wass_p(Y) \geq \dim Y$ follows.
\end{proof}

\subsection{The main Theorem and corollaries}

Now the proof of Theorem \ref{theo:size} is easy: we are given a compact metric space $X$,
and we want to prove 
\[\crit_\mathscr{P} \wass_p(Y) \geq \dim X.\]
If $\dim X=0$ there is nothing 
to prove; assume otherwise and choose arbitrary $s<s'<\dim X$. There is a bi-Lipschitz embedding of
an ultrametric space $X'$ into $X$ with $\dim X'\geq s'$ (by the Mendel-Naor ultrametric skeleton
 theorem)
and there is a co-Lipschitz embedding of a regular ultrametric space $Y$ into $X'$
with $\dim Y\geq s$ (by Theorem \ref{theo:structure}). Composing these embeddings
and applying the resulting map to measures, we get a co-Lipschitz embedding
$\wass_p(Y)\hookrightarrow \wass_p(X)$. This implies that 
\[\crit_\mathscr{P}\wass_p(X) \geq \crit_\mathscr{P}\wass_p(Y) = \dim Y \geq s\]
and since $s<\dim X$ is arbitrary, we finally get $\crit_\mathscr{P}\wass_p(X)\geq \dim X$.

The two corollaries then follow directly.
Assume $X$ is a compact space; we just saw that
$\crit_\mathscr{P} \wass_p(X)\geq \dim X$, and we proved in \cite{K:Hausdorff} that
$\crit_\mathscr{P} \wass_p(X)\leq \udim X$. If $\dim X=\udim X=d$, then we get
$\crit_\mathscr{P} \wass_p(X)=d$. If $\dim X>\udim X'$, we get 
$\crit_\mathscr{P} \wass_p(X)>\crit_\mathscr{P} \wass_{p'}(X')$, and
there cannot be any bi-Lipschitz (or even co-Lipschitz) map from the first Wasserstein
space to the second one.

\section{An alternative proof of Theorem \ref{theo:size}}\label{sec:alt}

We can prove Theorem \ref{theo:size} without the intermediate of ultrametric spaces,
by using instead a sequence of points with controlled distances.
\begin{lemm}\label{lem:sequence}
If $X$ is a metric space of Hausdorff dimension $d$, then for all
$d'<d$ there exist a constant $C$ and
a sequence of points $(q_i)$ in $X$ such that for all $i<j$ it holds
$d(q_i,q_j)\geq C i^{-1/d'}$.
\end{lemm}

\begin{proof}
This is a consequence of Frostman's Lemma. Given $d''\in(d',d)$, there
is a probability measure $\mu$ on $X$ such that $\mu(B(p,r))\leq C_1 r^{d''}$
for some constant $C_1$ and all $p$ in its support.
For all integer $i>0$, let $a_i = C_2 i^{-(1+\varepsilon)}$
where $\varepsilon$ will be chosen small afterward, and $C_2$ is such that
$\sum a_i =1$. Let $r_i= (a_i/C_1)^{\frac1{d''}}$

Choose $q_1\in \supp\mu$ arbitrarily; then $\mu(B(q_1,r_1))\leq a_1 <1$
so there is a $q_2$ outside $B(q_1,r_1)$. We construct recursively
$B_j= B(q_j, r_j)$ and $q_j$ outside $\cup_{i<j} B_i$. this is possible because
\[\mu(B_1\cup \dots B_{j-1})\leq a_1+\dots a_{j-1} <1 = \mu(X).\]

We then get that $d(q_i,q_j)$ is at least $r_i= C i^{-\frac{1+\varepsilon}{d''}}$
and we only have left to choose $\varepsilon$ and $d''$ appropriately.
\end{proof}

Let $d'<d$ be fixed, and $(q_i)$ be a sequence of points of $X$
as given by the lemma. Consider the map
\begin{align*}
\Phi : [0,1]^{\mathbb{N}} &\to \wass_p(X) \\
\bar x =(x_1,\dots) &\mapsto \sum_{i\geq 1} b_i x_i \delta_{q_{i+1}} 
  + (1- \sum_{i\geq 1} b_i x_i) \delta_{q_1}
\end{align*}
where $b_i=C_3i^{-(1+\varepsilon)}$ with $\varepsilon$ arbitrarily small
and $C_3=C_3(\varepsilon)$ is such that $\sum b_i\leq 1$.

Then we easily get the lower estimate
\[\dw_p^p(\Phi(\bar x),\Phi(\bar y)) \geq 
  \sum_{i\geq 1} b_i|x_i-y_i| \cdot C^p (i+1)^{-\frac{p}{d'}}.\]
Indeed, any transport plan from $\Phi(\bar x)$ to $\Phi(\bar y)$
must move a mass at least $b_i |x_i-y_i|$ from or to the point
$q_i$, thus this amount is moved by a distance at least $C (i+1)^{-\frac{1}{d'}}$.

We can identify $[0,1]^{\mathbb{N}}$ with any $\bcube((n^{-\alpha}))$
by suitable dilation along each coordinate.
Taking $\alpha=p/d'+1+\varepsilon$, we get that 
\[W_p^p(\Phi(\bar x),\Phi(\bar y)) \geq C_4 d(\bar x, \bar y)\]
where the distance on the right-hand side is obtained from
$\bcube((n^{-\alpha}))$ by the identification.

From \eqref{eq:BC} page \pageref{eq:BC}
we know that there is a probability measure $\mu$ on
$\bcube((n^{-\alpha}))$ such that 
\[\log \mu(B(\bar x, r) \leq C_5 r^{\frac{-1}{\alpha-1}} 
  = C_5 r^{\frac{-1}{p/d'+\varepsilon}}\]
(this is Frostman's Lemma, or rather what one proves to
bound from below the critical parameter of the Banach cube).

Consider the pushed forward measure $\Phi_\#\mu$ on $\wass_p(X)$:
for all $\bar x\in\bcube((n^{-\alpha}))$ we have 
\begin{align*}
\log \Phi_\#\mu(B(\Phi(\bar x),r)) &\leq \log \mu(B(\bar x, r^p/C_4)) \\
  &\leq C_6 r^{-\frac{p}{p/d'+\varepsilon}}
\end{align*}
This shows that the image of $\Phi$, and therefore $\wass_p(X)$ as well,
has $\mathscr{P}$-critical parameter at least
\[\frac{p}{\frac p{d'} + \varepsilon}\]
for all $\varepsilon >0$ and all $d'<d$. Letting $\varepsilon \to 0$
and $d'\to d$, we get that $\crit_\mathcal{P} \wass_p(X) \geq d$, as desired.

\begin{rema}
One could think that Lemma \ref{lem:sequence} should hold under a
condition on Minkowski dimension rather than Hausdorff dimension.
In next section an example is given showing that this is far from being true.
\end{rema}

\section{Concluding examples}\label{sec:examples}

\subsection{A large space with small parts}

To show that the Hausdorff dimension hypothesis in Lemma \ref{lem:sequence}
cannot easily be relaxed, let us prove the following.
\begin{prop}\label{prop:smallparts}
There is a compact metric space $X$ with lower Minkowski dimension equal to $1$,
all of whose proper closed subset have vanishing lower Minkowski dimension.
\end{prop}

This set $X$ is therefore ``large'' in the sense of lower Minkowski dimension,
but its closed parts are all ``small'' in the same sense. The example we shall construct 
has the additional
properties to have a Minkowski dimension (lower and upper dimension match)
and to be ultrametric.
Before proving the Proposition, let us see how it relates to Lemma \ref{lem:sequence}.

\begin{coro}
There is a compact metric space $X$ with positive lower Minkowski dimension $d$,
but such that for no $d'>0$ and no $C$ does it exists a sequence $(q_i)$ in $X$ with
$d(q_i,q_j)\geq C i^{-1/d'}$ for all $i<j$.
\end{coro}

\begin{proof}
If a space $Y$ contains a sequence $(q_i)$ with
$d(q_i,q_j)\geq C i^{-1/d'}$ for all $i<j$, then it has lower Minkowski dimension
at least $d'$, as for any $\varepsilon$ one needs at least
 $N=\lfloor(C/\varepsilon)^{d'}\rfloor$
sets of diameter $\varepsilon$ to cover the the $N$ first points
in the sequence.

Now, if the space $X$ given by Proposition \ref{prop:smallparts} had such a sequence,
its proper closed subset $Y = X \setminus B(q_1,C)$ would contain the sequence
$(q_i)_{i>1}$ and therefore have lower Minkowski dimension at least $d'$, a 
contradiction.
\end{proof}

\begin{proof}[Proof of Proposition \ref{prop:smallparts}]
We construct $X$ as an ultrametric space. Its SRT $T$ is given 
in terms of two sequences $(a_n)$, $(h_n)$ to be chosen suitably afterward.
Its vertices are numbered $1,2,\dots$ by a breadth-first search,
and the vertex $n$ has $a_n$ children. In other words, $1$ is the root,
$2,\dots,1+a_1$ are the depth-$1$ vertices, $a_1+2,a_1+3,\dots,a_1+a_2+1$
are the children of $2$, and so on. We assume $a_n>1$ for all $n$,
so that the SRT has no leaf and $X$ has the topology of a Cantor set.
Then $h_n$, assumed to be decreasing, is the height of $n$.

Now, let $A_n=a_1+a_2+\dots+a_n-n+1$ be the number of branches of $T$
at height slightly below $h_n$. We set $a_1=2$, $a_{n+1}=A_n+1$ and
$h_1=1$, $h_{n+1}=1/A_n$. In that way, $A_n=2^n$ and $h_n = 2^{-n+1}$.

For any $\varepsilon$, let $n$ be such that $2^{-n+1}>\varepsilon\geqslant 2^{-n}$:
one needs $2^n$ balls of radius $\varepsilon$ to cover $X$, and
$2^n$ is between $1/\varepsilon$ and $2/\varepsilon$. Therefore, $X$ has Minkowski 
dimension $1$.

To prove that proper closed subset of $X$ have vanishing lower Minkowski dimension,
we are reduced to consider $X\setminus B$ where $B$ is some ball,
the set of descendants of vertex $i$ say. 
Let $k$ be a descendant of $i$: it has $A_{k-1}+1=2^{k-1}+1$ children,
numbered from $n=k+1+2^{k-1}$ to $m=k+1+2^k$. Then $X\setminus B$
can be covered by less than $A_n$ balls of radius $\varepsilon=2^{-m}$. Since for large
$k$, $A_n$ has the order of $\varepsilon^{-1/2}$, $X\setminus B$
has lower Minkowski dimension at most $1/2$.

But if $k$ is taken large enough, it has arbitrarily many successive
siblings $k+1,k+2,\dots, k+N$ all of which are descendants of $i$.
Then we can cover $X\setminus B$
by $A_n$ balls of radius $\varepsilon=2^{-M}$ with
\begin{align*}
M &= n+A_{k-1}+1+A_k+1+\dots+A_{k+N}+1 \\
  &\geq 2^{k-1}+2^{k-1}+2^k+2^{k+1}\dots+2^{k+n-1}\\
  &\geq 2^{k+N}-2^k = 2^{k-1}(2^{N+1}-2)
\end{align*}
For any $d>0$ and large enough $k$, we get that $A_n$ is
far less than $\varepsilon^{-d}$. Therefore, the lower Minkowski dimension
of $X\setminus B$ is zero.
\end{proof}

\subsection{A small space with large Wasserstein space}
Last, we prove Proposition \ref{prop:counterex}. The example is constructed
to have infinite Minkowski dimension (the number of branches above height $\varepsilon$
in its SRT grows very fast when $\varepsilon\to 0$) but
small ``complexity'' (its SRT has countably many ends).

Let $T$ be the SRT such that
the root has two children: a leaf $w_1$ (thus $h(w_1)=0$), and $v_2$ which
has two children: a leaf $w_2$, and $v_3$ which has two children and so on;
and such that $h(o)=1$ and $h(v_n)=(1+\ln n)^{-1}$. The ultrametric space $X$ defined by $T$
is a sequence $\{w_1,w_2,\dots\}$ together with an accumulation point $w_\infty$,
and $d(w_i,w_j)=(1+\ln i)^{-1}$ when $1\leq i<j\leq \infty$.

For any $p>1$ and any $\mu,\nu\in\pr(X)$, formula \eqref{eq:coordinates}
shows that
\[\dw_p(\mu,\nu)^p\geq \sum_{n\geq 1}\frac1{(1+\ln n)^p} |\mu(\{w_n\})-\nu(\{w_n\})|.\]
Fix any $\varepsilon>0$, and
let $(b_n)$ be the sequence of sum $1$ such that $b_n=Cn^{-(1+\varepsilon)}$ for some
$C$ and all $n$. The map from $[0,1]^\mathbb{N}$ that sends
$(m_n)$ to the measure $\mu$ such that $\mu(\{w_n\})=m_n b_n$ is therefore co-Lipschitz
from $\bcube((c_n))$ to $(\pr(X),\dw_p^p)$ with
\[a_n=\frac{b_n}{(1+\ln n)^p}=\Theta\left(\frac1{n^{1+\varepsilon}(1+\ln n)^p}\right)
  =\Omega\left(\frac1{n^{1+2\varepsilon}}\right).\]
In particular, one can restrict this map to a co-Lipschitz map with domain
$\bcube((n^{-(1+2\varepsilon)}))$ whose critical parameter is $1/2\varepsilon$.
It follows that $\crit_\mathscr{P} \wass_p(X)\geq \frac{p}{2\varepsilon}$, and this holds
for all $\varepsilon>0$.

In conclusion, $X$ is a countable ultrametric space whose Wasserstein
space has infinite power-exponential critical parameter,
as claimed.

It seems plausible that (maybe at least for ultrametric spaces) the critical parameter
of $\wass_p(X)$ is bounded below by the (lower or upper?) Minkowski dimension of $X$.
The strongest conjecture would be that $\crit_\mathscr{P}\wass_p(X)=\udim X$ for all
compact metric space $X$.
However we do not know how to transform an ultrametric space into one where explicit computations are 
possible without loosing too much of its
Minkowski dimension, and general ultrametric spaces seem difficult to handle in wide generality.
Moreover, we do not know whether there is an Ultrametric Skeleton Theorem with respect to
Minkowski dimension (see Question 1.11 in \cite{Mendel-Naor:ultrametric}).

\bibliographystyle{smfalpha}
\bibliography{biblio}

\end{document}